\documentclass[12pt]{amsart}
\usepackage{amssymb, amsthm, amsmath}
\newtheorem{question}{Question}[section]
\newtheorem{theorem}[question]{Theorem}

\newtheorem{corollary}[question]{Corollary}
\newtheorem{definition}[question]{Definition}

\newtheorem{proposition}[question]{Proposition}
\newtheorem{example}[question]{Example}

\title{Infinite games and cardinal properties of topological spaces}

\author{
Angelo Bella}

\address{
Department of Mathematics and Computer Science \\
University of Catania \\
Citt\'a universitaria\\  
viale A. Doria 6 \\
95125 Catania, Italy}

\email{bella@dmi.unict.it}

\author{
Santi Spadaro}

\address{
Department of Mathematics and Statistics\\
Faculty of Science and Engineering\\
York University \\
Toronto, ON\\
M3J 1P3 Canada}

\curraddr{Institute of Mathematics, Silesian University in Opava, Na Rybn\' i\v cku 626/1, 746 01 Opava, Czech Republic}

\email{santispadaro@yahoo.com}

\thanks{The second named author was supported by an INdAM-Cofund Outgoing fellowship.}



\begin{document}

\maketitle

\begin{abstract}
Inspired by work of Scheepers and Tall, we use properties defined by topological games to provide bounds for the cardinality of topological spaces. We obtain a partial answer to an old question of Bell, Ginsburg and Woods regarding the cardinality of weakly Lindel\"of first-countable regular spaces and answer a question recently asked by Babinkostova, Pansera and Scheepers. In the second part of the paper we study a game-theoretic version of cellularity, a special case of which has been introduced by Aurichi. We obtain a game-theoretic proof of Shapirovskii's bound for the number of regular open sets in an (almost) regular space and give a partial answer to a natural question about the productivity of a game strengthening of the countable chain condition that was introduced by Aurichi. As a final application of our results we prove that the Hajnal-Juh\'asz bound for the cardinality of a first-countable ccc Hausdorff space is true for almost regular (non-Hausdorff) spaces.
\end{abstract}

\section{Introduction}

One of the main inspirations for our work is the celebrated theorem of Arhangel'skii, which in 1969 solved a forty years old problem of Alexandroff and Urysohn.

\begin{theorem} \label{arhanthm}
(Arhangel'skii's Theorem) If $X$ is a first-countable Lindel\"of space, then $|X| \leq 2^{\aleph_0}$.
\end{theorem}

The importance of this result is witnessed by the wealth and variety of generalizations that have been offered of it.

Very soon after the publication of his paper, Arhangel'skii himself asked whether Theorem $\ref{arhanthm}$ remains true if the condition of first-countability is relaxed to the requirement that all points are $G_\delta$. This turned out to be a difficult problem and it was eventually solved by Saharon Shelah in \cite{Sh} in the negative. Later on, Isaac Gorelic found a simpler counterexample. He constructed a model of ZFC where CH holds, $2^{\aleph_1}$ can be arbitrarily large and there exists a zero-dimensional Lindel\"of space with points $G_\delta$ and cardinality $2^{\aleph_1}$. However, it is still an open question whether the cardinality of a Lindel\"of Hausdorff space with points $G_\delta$ is always bounded by $2^{\aleph_1}$. Arhangel'skii proved that the cardinality of every Lindel\"of $T_1$ space of countable pseudocharacter is less than the first measurable cardinal, while Juh\'asz constructed examples of Lindel\"of $T_1$ spaces with countable pseudocharacter of arbitrarily large cardinality below the first measurable. However, it is still unknown whether the cardinality of every Lindel\"of first-countable $T_1$ space is bounded by the continuum.

Scheepers and Tall recently provided a new partial positive solution to Arhangel'skii's problem by using infinite games. To present it, we need to fix some notation. Let $\alpha $ be an ordinal and  $\mathcal{A}$ and  $\mathcal{B}$ be collection of sets.   The symbol $G^\alpha
_1(\mathcal{A}, \mathcal{B})$ (respectively, $G^\alpha _{fin}(\mathcal{A}, \mathcal{B})$) denotes  the game of length $\alpha $ played by two players 1 and 2  in the
following way:   at  the $\beta$-inning player 1 choose $\mathcal{U}_\beta \in \mathcal{A}$ and player 2 responds by taking an element $V_\beta \in \mathcal{U}_\beta$ (respectively, a finite set $\mathcal{V}_\beta \subseteq \mathcal{U}_\beta $). Player 2 wins if and only if   $\{V_\beta
:\beta<\alpha \}\in \mathcal{B}$ (respectively, $\bigcup \{\mathcal{V}_\beta:\beta<\alpha \}\in \mathcal{B}$).

Given a topological space $X$, we denote by $\mathcal{O}_X$ the set of all open covers of $X$. In \cite{ST} Scheepers and Tall focus on the game $G^{\omega_1}_1(\mathcal{O}_X, \mathcal{O}_X)$. It is evident that if the cardinality of $X$ does not exceed $\aleph_1$, then player two has a winning strategy in $G^{\omega_1}_1(\mathcal{O}_X, \mathcal{O}_X)$. It is also not hard to see that if $X=2^{\omega_1}$ then player one has a winning strategy in $G^{\omega_1}_1(\mathcal{O}_X, \mathcal{O}_X)$. One of the main results of Scheeper and Tall's paper \cite{ST} is the following variant of Arhangel'skii's Theorem.

\begin{theorem} \label{scheepersandtall}
(Scheepers and Tall, \cite{ST}) Let $X$ be a space of countable pseudocharacter such that player two has a winning strategy in $G^{\omega_1}_1(\mathcal{O}_X, \mathcal{O}_X)$. Then $|X| \leq 2^{\aleph_0}$.
\end{theorem}

Another important generalization of Arhangel'skii's Theorem was proposed by Bell, Ginsburg and Woods in 1978. Recall that a space $X$ is \emph{weakly Lindel\"of} if every open cover has a countable subcollection whose union is dense in $X$. The class of weakly Lindel\"of spaces is much more general than the class of Lindel\"of spaces and includes the class of all spaces with the countable chain condition.

\begin{theorem}
(Bell, Ginsburg and Woods, \cite{BGW}) Let $X$ be a weakly Lindel\"of first-countable normal space. Then $|X| \leq 2^{\aleph_0}$.
\end{theorem}

The question of whether normality can be relaxed to regularity in the above theorem is still open.

\begin{question} \label{bellginsburgwoods}
(\cite{BGW}) Is the cardinality of every regular first-countable weakly Lindel\"of space bounded by the continuum?
\end{question}

We will offer a partial answer to this question by using a game-theoretic variant of weak Lindel\"ofness. 

Other attempts have focused on extending Arhangel'skii's Theorem to the realm of non-regular spaces. 

Recall that a space $X$ is called \emph{$H$-closed} if every open cover has a finite subfamily whose union is dense in $X$ and is called \emph{almost Lindel\"of} if for every open cover $\mathcal{U}$ of $X$ there is a countable collection $\mathcal{V} \subset \mathcal{U}$ such that $\bigcup \{\overline{V}: V \in \mathcal{V} \}=X$. It is easy to realize that every regular almost Lindel\"of space is Lindel\"of and every regular $H$-closed space is compact, but these are strict weakenings of Lindel\"ofness in the non-regular realm. Obviously, every $H$-closed space is almost-Lindel\"of. Recall that a space $X$ is called \emph{Urysohn} if for every pair of distinct points $x,y \in X$ there are open neighbourhoods $U$ of $x$ and $V$ of $y$ such that $\overline{U} \cap \overline{V}=\emptyset$.

\begin{theorem} \label{bellacammaroto}
(Bella and Cammaroto, \cite{BC}) Let $X$ be an almost Lindel\"of Urysohn first-countable space. Then $|X| \leq 2^{\aleph_0}$.
\end{theorem}

\begin{theorem} \label{gryzlov}
(Gryzlov, \cite{G}) If $X$ is a first-countable $H$-closed space then $|X| \leq 2^{\aleph_0}$.
\end{theorem}

The obvious common generalization of Theorems $\ref{bellacammaroto}$ and $\ref{gryzlov}$ was later disproved by Bella ad Yaschenko in \cite{BY}, where the authors construct first-countable Hausdorff almost Lindel\"of spaces of arbitrarily large cardinality.

By considering a game-theoretic variant of the almost Lindel\"of property we will be able to relax the first-countability condition in Theorem $\ref{bellacammaroto}$ while Theorem $\ref{gryzlov}$ will be key in disproving a conjecture of Babinkostova, Pansera and Scheepers regarding another game-theoretic strengthening of the same property.

The interested reader can find more on variants of Arhangel'skii's Theorem in Richard Hodel's survey paper \cite{H}.

One of the most important cardinal invariants in topology is the \emph{cellularity}, that is the supremum of sizes of families of pairwise disjoint non-empty open sets in a topological space. This cardinal invariant is featured in many cardinal bounds in Juh\'asz's book \cite{J} and has stimulated much research on the border between topology and set theory. 

Leandro Aurichi \cite{A} recently introduced a natural game strengthening of countable cellularity and discussed its productive behavior.  Considering higher cardinal versions of his game will allow us to give a game-theoretic proof of Shapirovskii's bound on the number of regular open sets, improve Aurichi's results on products and show that the Hajnal-Juh\'asz bound on the cardinality of a first-countable ccc space is true for almost regular non-Hausdorff spaces. We finish by showing that this last bound is not true for $T_1$ spaces, even if the countable chain condition is strengthened to Aurichi's game version of it. 

For notation and terminology we refer to \cite{E}.

\section{Variants of Arhangel'skii's Theorems}

Given a topological space $X$ we denote by $\mathcal{D}_X$ the collection of all families of open sets $\mathcal{U}$ such that $\bigcup \mathcal{U}$ is dense in $X$ and by $\overline{\mathcal{O}}_X$ the collection of all families of open sets $U$ such that $\{\overline{U}: U \in \mathcal{U} \}$ covers $X$.

Recall that a space $X$ is said to be \emph{almost regular} if the collection of all non-empty regular closed sets is a $\pi$-network for $X$, that is, for every non-empty open set $U \subset X$ there is a non-empty open set $V$ such that $\overline{V} \subset U$. Every regular space is clearly almost regular.

The closed pseudocharacter of $X$ ($\psi_c(X)$)
 is defined as the least cardinal $\kappa$ such that for every $x \in X$ there is a family $\mathcal{U}$ of open neighbourhoods of $X$ such that $\bigcap \{\overline{V}: V \in \mathcal{V} \}=\{x\}$.

\begin{theorem} \label{almostgamethm}
Let $X$ be a space of countable closed pseudocharacter (in particular, a first-countable Hausdorff space). If player two has a winning strategy in $G^{\omega_1}_1(\mathcal{O}_X, \overline{\mathcal{O}}_X)$ then $|X| \leq 2^{\aleph_0}$.
\end{theorem}

\begin{theorem}  \label{weaklygamethm}
Let $X$ be a first-countable almost regular space. If player two has a winning strategy in $G^{\omega_1}_1(\mathcal{O}_X, \mathcal{D}_X)$ then $|X| \leq 2^{\aleph_0}$.
\end{theorem}

The proofs of Theorems $\ref{almostgamethm}$ and $\ref{weaklygamethm}$ have been clumped together since they are very similar.

\begin{proof}
For every $x \in X$ fix a family $\{U_\alpha(x): \alpha < \mathfrak{c}\}$ at $x$ witnessing the countable closed pseudocharacter (respectively, the countable local character) of the point $x$ in $X$. Let $F$ be a winning strategy for player II in the game $G^{\omega_1}_1(\mathcal{O}_X, \overline{\mathcal{O}}_X)$ (respectively, $G^{\omega_1}_1(\mathcal{O}, \mathcal{D})$).

\vspace{.1in}

\noindent {\bf Claim 1}
Let $\alpha < \omega_1$ and $(\mathcal{U}_\beta: \beta < \alpha )$ be a sequence of open covers. For every neighbourhood $U$ of $x$ there is an open cover $\mathcal{U}$ such that:

$$F((\mathcal{U}_\beta: \beta < \alpha )^\frown (\mathcal{U})) =U$$

\begin{proof}[Proof of Claim 1]

Assume the contrary, and for every $x \in X$ fix a neighbourhood $U_x$ of $x$, such that, for every open cover $\mathcal{U}$, we have:

$$F((\mathcal{U}_\beta: \beta < \alpha)^\frown (\mathcal{U})) \neq U_x$$

The set $\mathcal{V}=\{U_x: x \in X \}$ is an open cover of $X$, and thus we may find $y \in X$ such that $F((\mathcal{U}_\beta: \beta < \alpha )^\frown (\mathcal{V})) = U_y$, but that is a contradiction.
 \renewcommand{\qedsymbol}{$\triangle$}
\end{proof}

Use the claim to choose $x_\emptyset \in X$ so that for each open neighbourhood $U$ of $x_\emptyset$ there is an open cover $\mathcal{U}$ such that $F((\mathcal{U})) = U$. Then, for each $n<\omega$ choose an open cover $\mathcal{U}_{\{(0, n)\}}$ such that $F((\mathcal{U}_{\{(0, n) \}}) =U_n(x_\emptyset)$.

For each $n_0 < \omega$ choose $x_{\{(0, n_0)\}} \in X$ satisfying Claim 1 for the open cover $\mathcal{U}_{\{(0, n_0)\}}$. Then, for each $n_0, n < \omega$ choose an open cover $\mathcal{U}_{\{(0, n_0), (1, n)\}}$ with:

$$F((\mathcal{U}_{\{(0, n_0)\}}, \mathcal{U}_{\{(0, n_0), (1, n)\}})) = U_n(x_{\{(0, n_0)\}}).$$

Then for each $n_0, n_1 < \omega$ choose $x_{\{(0, n_0), (1, n_1)\}} \in X$ as in Claim 1 for the sequence of open covers $(\mathcal{U}_{\{(0, n_0)\}}, \mathcal{U}_{\{(0, n_0), (1, n_1)\}})$. For each $n < \omega$ choose an open cover $\mathcal{U}_{\{(0, n_0), (1, n_1), (2, n)\}}$ such that 

$$F((\mathcal{U}_{\{(0, n_0)\}}, \mathcal{U}_{\{(0, n_0), (1, n_1)\}}, \mathcal{U}_{\{(0, n_0), (1, n_1), (2, n)\}})) = U_n (x_{\{(0, n_0), (1, n_1)\}}),$$

and so on. 

Now let $\alpha < \omega_1$ and suppose that for every $f \in \bigcup_{\beta< \alpha} \omega^\alpha$ we have chosen points $x_f \in X$ and open covers $\mathcal{U}_f$ such that for all $\beta < \alpha$ and for every $f \in \omega^\beta$ and $n<\omega$ we have:

$$F((\mathcal{U}_{f \upharpoonright \gamma}: \gamma < \beta )^\frown (\mathcal{U}_{f \cup \{(\beta, n)\}})) =U_n(x_f)$$.

Consider $f \in \omega^\alpha$. Let $x_f \in X$ be the point guaranteed by applying Claim 1 to the sequence $(U_{f \upharpoonright \beta}: \beta < \alpha)$ of open covers. Then, for each $n<\omega$ choose an open cover $\mathcal{U}_{f \cup \{(\alpha, n)\}}$ such that:

$$F((\mathcal{U}_{f \upharpoonright \beta}: \beta < \alpha)^\frown (\mathcal{U}_{f \cup \{(\alpha, n)\}})) \subset U_n(x_f)$$

At the end of the induction we will have chosen, for each $f \in \bigcup_{\alpha < \omega_1} \omega^\alpha$ a point $x_f \in X$ and an open cover $\mathcal{U}_f$ such that for each such $f$ and each $n<\omega$, if $f \in \omega^\chi$ then

$$F((\mathcal{U}_{f \upharpoonright \beta}: \beta < \chi)^\frown (\mathcal{U}_{f \cup \{(\chi, n)\}})) = U_n(x_f).$$

Let $D=\{x_f: f \in \bigcup_{\alpha < \omega_1} \omega^\alpha\}$ and note that $|D| \leq \mathfrak{c}$.

\vspace{.1in}

\noindent {\bf Claim 2.} The set $D$ coincides with $X$ (respectively, is dense in $X$).

\begin{proof}[Proof of Claim 2]
Suppose not. Then we can find a point $p \in X \setminus D$ (respectively, an open set $V$ such that $\overline{V} \subset X \setminus \overline{D}$). Choose $n_0$ such that $p \notin U_{n_0}$ (respectively, $U_{n_0}(x_\emptyset) \cap \overline{V}=\emptyset$). Player I plays with $\mathcal{U}_{\{(0, n_0)\}}$ and player II responds by means of the strategy $F$. Suppose player I has played open covers $( \mathcal{U}_{f \upharpoonright \beta}: \beta < \alpha)$ for some $f \in \omega^\alpha$ such that $f(0)=n_0$. Choose $n_\alpha$ such that $p \notin U_{n_\alpha}$ (respectively, $U_{n_\alpha}(x_f) \cap \overline{V}=\emptyset$) and let player I play $\mathcal{U}_{f \cup \{(\alpha, n_\alpha)\}}$. In this way we build an $F$-play which is lost by player two. But this contradicts the fact that $F$ is a winning strategy for player two.
 \renewcommand{\qedsymbol}{$\triangle$}
\end{proof}

Since we clearly have $|D| \leq 2^{\aleph_0}$ the proof of Theorem 1 is concluded. To finish the proof of Theorem 2 it suffices to observe that, since $X$ is first countable and $D$ is dense, we also have that $|X| \leq 2^{\aleph_0}$.

\end{proof}

\begin{corollary}
[CH] Suppose $X$ is a space of countable closed pseudocharacter. Then player II has a winning strategy in $G^{\omega_1}_1(\mathcal{O}_X, \overline{\mathcal{O}}_X)$ if and only if $|X| \leq \aleph_1$.
\end{corollary}

\begin{corollary}
[CH] Suppose $X$ is a first-countable almost regular space. Then player II has a winning strategy in $G^{\omega_1}_1(\mathcal{O}_X, \mathcal{D}_X)$ if and only if $|X| \leq \aleph_1$.
\end{corollary}

If player two can always win the respective games in $\omega$ many moves, then it's clear that the set $D$ constructed in the proof is actually countable. This allows us to give alternative proofs to two theorems presented in \cite{BPS}.

\begin{theorem} \label{countablethm}
Let $X$ be a space of countable closed pseudocharacter. If player two has a winning strategy in the game $G^\omega_1(\mathcal{O}, \overline{\mathcal{O}})$ then $X$ is countable.
\end{theorem}

\begin{theorem} 
Let $X$ be a first.countable almost regular game. If player two has a winning strategy in the game $G^\omega_1(\mathcal{O}, \mathcal{D})$ then $X$ is separable,
\end{theorem}

A further inspection in the proofs of Theorems $\ref{scheepersandtall}$ and $\ref{almostgamethm}$ shows that they can be formulated in a more general way.

\begin{theorem}
Let $X$ be a space such that $\psi(X) \leq 2^{\aleph_0}$. If player two has a winning strategy in $G^{\omega_1}_1(\mathcal{O}_X, \mathcal{O}_X)$ then $|X| \leq 2^{\aleph_0}$.
\end{theorem}

\begin{theorem}
Let $X$ be a space such that $\psi_c(X) \leq 2^{\aleph_0}$. If player two has a winning strategy in $G^{\omega_1}_1(\mathcal{O}_X, \overline{\mathcal{O}}_X)$ then $|X| \leq 2^{\aleph_0}$.
\end{theorem}

Also, the following theorem can be proved along the lines of the proof of Theorem $\ref{almostgamethm}$.

\begin{theorem}
Let $X$ be a regular sequential space such that $\chi(X) \leq 2^{\aleph_0}$. If player two has a winning strategy in $G^{\omega_1}_1(\mathcal{O}, \mathcal{D})$ then $|X| \leq 2^{\aleph_0}$.
\end{theorem}

Theorem $\ref{scheepersandtall}$ is no longer true if we assume that player one does not have a winning strategy in $G^{\omega_1}_1(\mathcal{O}_X, \mathcal{O}_X)$. Indeed Scheepers and Tall point out in \cite{ST} that if $X$ is Gorelic's space, then $X$ has the Rothberger property, and hence, by a result of Pawlikowski \cite{P} player one does not have a winning strategy even in the shorter game $G^\omega_1(\mathcal{O}_X, \mathcal{O}_X)$. Since, for regular spaces, Theorem $\ref{scheepersandtall}$ and Theorem $\ref{almostgamethm}$ coincide, we see that also in Theorem $\ref{almostgamethm}$ \emph{player two has a winning strategy} cannot be weakened to \emph{player one does not have a winning strategy}.

Theorem $\ref{scheepersandtall}$ suggests the following natural questions:

\begin{question} \label{questfin}
Let $X$ be a space of countable closed pseudocharacter and assume that player two has a winning strategy in $G^{\omega_1}_{fin}(\mathcal{O}_X, \mathcal{O}_X)$. Is it true that $|X| \leq 2^{\aleph_0}$?
\end{question}

As a mild motivation towards a positive answer to the above question, note that if $X$ is a $\sigma$-compact space then player two has a winning strategy in $G^{\omega}_{fin}(\mathcal{O}_X, \mathcal{O}_X)$.

\begin{question}
Let $X$ be a first-countable space in which player one does not have a winning strategy in $G^{\omega_1}_1(\mathcal{O}, \mathcal{O})$. Is it true that $|X| \leq 2^{\aleph_0}$?
\end{question}

In reference to the above question, note that the fact that player one does not have a winning strategy in $G^{\omega_1}_1(\mathcal{O}_X, \mathcal{O}_X)$ appears to be much stronger than $X$ having Lindel\"of number $\leq \omega_1$.

Theorem $\ref{almostgamethm}$ suggests the following question:

\begin{question}
Let $X$ be a first countable Hausdorff space such that player one does not have a winning strategy in $G^{\omega}_1(\mathcal{O}_X, \overline{\mathcal{O}}_X)$. Is it true that $|X| \leq 2^{\aleph_0}$?
\end{question}

If player one does not have a winning strategy in $G^{\omega_1}_1(\mathcal{O}_X, \overline{\mathcal{O}}_X)$ then $X$ is almost Lindel\"of, so the Bella-Cammaroto theorem implies that the above question has a positive answer in the realm of Urysohn spaces.

The analogous of Question $\ref{questfin}$ for Theorem $\ref{almostgamethm}$ has a negative answer, as we are now going to show.

\begin{definition}
A subset $X$ of the space $Y$ is relative $H$-closed in $Y$ if for every open cover $\mathcal{U}$ of $Y$ there is a finite subfamily $\mathcal{U}' \subset \mathcal{U}$ satisfying $X \subseteq \overline{\mathcal{U}'}$.
\end{definition}

\begin{theorem} \label{thmdenseH}
If a space $X$ is the union of countably many sets, each of which is relatively $H$-closed in $X$, then player two has a winning strategy in the game $G^\omega_{fin} (\mathcal{O}, \overline{\mathcal{O}})$.
\end{theorem}

\begin{proof}
Let $\{Y_n: n < \omega \}$ be a cover of $X$ by relatively $H$-closed sets. To define a winning strategy for player two in $G^\omega_1(\mathcal{O}_X, \overline{\mathcal{O}}_X)$, assuming player one chooses the open cover $\mathcal{U}$ at the $n$-th inning, let player two choose a finite collection $\mathcal{V} \subset \mathcal{O}$ such that $Y_n \subset \bigcup \mathcal{V}$.
\end{proof}

In a similar way, we can prove the following theorem.

\begin{theorem} 
(\cite{BPS}) If a space $X$ contains a dense subspace which is the union of countably many relatively $H$-closed subsets, then player two has a winning strategy in $G^\omega_{fin} (\mathcal{O}_X, \overline{\mathcal{O}}_X)$.
\end{theorem}

In Theorem $\ref{weaklygamethm}$ we cannot relax first-countability with countable tightness and countable pseudocharacter.

\begin{example} \label{examplepseudo}
For every cardinal $\kappa$ there is a space $X$ such that $|X|=\kappa$ and player two has a winning strategy in $G^{\omega_1}(\mathcal{O}, \mathcal{D})$.
\end{example}

\begin{proof}
In \cite{BGW} the authors constructed, for every cardinal $\kappa$, a space $X$ with the following properties:

\begin{enumerate}
\item $|X|=\kappa$.
\item $X$ has countable pseudocharacter and countable tightness.
\item There is a countable subset $Y \subset X$ such that for every open set $U$ containing $Y$, $X \setminus \overline{U}$ is countable.
\end{enumerate}

To see that player two has a winning strategy in $G^{\omega_1}_1(\mathcal{O}_X, \mathcal{D}_X)$ let $\{y_n: n < \omega \}$ be an enumeration of $Y$. For every $n<\omega$, if player one plays open cover $\mathcal{U}_n$ at the $n$-th inning, let player two choose an open set $U_n \in \mathcal{U}_n$ such that $x_n \in U_n$. Let $U=\bigcup_{n<\omega} U_n$ and $\{z_n: n < \omega \}$ enumerate $X \setminus \overline{U}$. For every $n<\omega$, if player one plays open cover $\mathcal{V}_n$ at inning $\omega+n$, let player two choose $V_n \in \mathcal{V}_n$ such that $z_n \in V_n$. Then $\{U_n: n <\omega \} \cup \{V_n: n <\omega \}$ has dense union in $X$ and thus player two wins the game if he plays according to the strategy described above.
\end{proof}

In Theorem $\ref{weaklygamethm}$ almost regularity cannot be dropped.

\begin{example}
For every cardinal $\kappa$ there is a first-countable Hausdorff space $X$ such that player two has a winning strategy in $G^\omega_1(\mathcal{O}_X, \mathcal{D}_X)$.
\end{example}

\begin{proof}
In \cite{BGW} the authors construct, for every cardinal $\kappa$, a first-countable space $X$ such that $|X|=\kappa$ and there is a subspace $Y \subset X$ such that every open set containing $Y$ is dense in $X$. In a similar way as in the proof of Example $\ref{examplepseudo}$, one proves that player two has a winning strategy in $G^\omega_1(\mathcal{O}_X, \mathcal{D}_X)$.
\end{proof}

We remarked that the assumption of countable closed pseudocharacter in Theorem $\ref{almostgamethm}$ is satisfied in the case of a Hausdorff first-countable space. However the closed pseudocharacter of $T_1$ first-countable spaces can be arbitrarily large, and in fact Sakai showed in \cite{S} an example of a first-countable $T_1$ space $X$ of arbitrarily large cardinality such that player two has a winning strategy in $G^\omega_1(\mathcal{O}_X, \overline{\mathcal{O}_X})$. 

\begin{example} \label{sakaiexample}
(Sakai, \cite{S}) There is a $T_1$ almost regular first countable space of arbitrarily large cardinality $X$ where player two has a winning strategy in $G^\omega_1(\mathcal{O}_X, \overline{\mathcal{O}}_X)$.
\end{example}

\begin{proof}
Let $\kappa$ be a cardinal and let $A$ and $B$ be disjoint sets such that $|A|=\omega$ and $|B|=\kappa$. Let $X=A \cup B$, where we declare each point of $A$ to be isolated and declare a basic neighbourhood of a point $x \in B$ to be of the form $\{x\} \cup C$ where $C$ is cofinite in $A$.

This is clearly a $T_1$ first-countable space and the closure of every open neighbourhood of a point of $B$ in $X$ is cofinite, and thus player two has a winning strategy in $G^\omega_1(\mathcal{O}_X, \overline{\mathcal{O}}_X)$. Moreover, $X$ is an almost regular space, because every point of $A$ is isolated and every open set hits $A$.
\end{proof}

In Question 23 of \cite{BPS}, the authors ask whether, in the Hausdorff realm, the space $X$ has a $\sigma$-H-closed dense subset if and only if player two having a winning strategy in $G^\omega_{fin}(\mathcal{O}_X, \mathcal{D}_X)$. We now present a negative answer to this question.

\begin{example}
There is a space $X$ such that player two has a winning strategy in $G^\omega_{fin}(\mathcal{O}_X, \mathcal{D}_X)$ but no dense subset of $X$ is the union of countably many $H$-closed subspaces.
\end{example}

\begin{proof}
In \cite{BY}, Yaschenko and the first author construct, for every measurable cardinal $\kappa$, a first-countable Hausdorff space $X$ such that $X$ is the union of countably many \emph{relatively} $H$-closed subsets and $|X|=\kappa$. By Theorem $\ref{thmdenseH}$ player two has a winning strategy in $G^\omega_{fin}(\mathcal{O}_X, \overline{\mathcal{O}}_X)$ (and hence also in $G^\omega_{fin}(\mathcal{O}_X, \mathcal{D}_X)$). However, no dense subset of $X$ is $\sigma$-$H$-closed. To see that, note that, since every first-countable $H$-closed space has cardinality bounded by the continuum, if a space has a $\sigma$-$H$-closed dense subspace then it also has density at most continuum. But a first-countable space of density bounded by the continuum also has cardinality bounded by the continuum.
\end{proof}

We finish this section by suggesting another possible way to give a partial answer to Question $\ref{bellginsburgwoods}$. Since the fact that player two has a winning strategy in $G^\omega_{fin}(\mathcal{O}, \mathcal{D})$ seems to be much stronger than the weak Lindel\"of property, it appears reasonable to ask the following question.

\begin{question}
Let $X$ be a first-countable regular space and assume that player two has a winning strategy in the game $G^\omega_{fin} (\mathcal{O}_X, \mathcal{D}_X)$. Is it true that $|X| \leq \mathfrak{c}$?
\end{question}

\section{The cellular-open game and Shapirovskii's bound on the weight}

Given a topological space $X$ we denote by $\mathcal{C}_X$ the family of all maximal pairwise disjoint open families in $X$. The game $G^\tau_1(\mathcal{C}_X, \mathcal{D}_X)$ was introduced by Aurichi in \cite{A} for the case $\tau=\omega$.

Note that there is a clear connection between this game and the games studied in the previous section.

\begin{proposition} \label{link}
If player two has a winning strategy in the game $G^\tau_1(\mathcal{C}, \mathcal{D})$ then it also has a winning strategy in the game $G^\tau_1(\mathcal{O}, \mathcal{D})$.
\end{proposition}

\begin{proof}
Indeed, let $\mathcal{U}_\alpha$ be the open cover played by player I at the $\alpha$-th winning. Refine $\mathcal{U}_\alpha$ to a maximal cellular family $\mathcal{C}_\alpha$. Player II uses his strategy to pick an open set $U_\alpha \in \mathcal{C}_\alpha$ so that $\bigcup \{U_\alpha: \alpha < \tau\}$ is dense in $X$.
\end{proof}

If $X$ is a space such that player two has a winning strategy in $G^\omega_1(\mathcal{C}_X, \mathcal{D}_X)$, then certainly $X$ is ccc. However, the game-theoretic property appears to be much stronger than the ccc, and thus it is reasonable to expect a better behaviour with respect to the various topological operations. For example, we may ask the following question.

\begin{question} \label{questprod}
Let $\{X_i: i \in I\}$ be a family of spaces such that player two has a winning strategy in $G^\omega_1(\mathcal{C}_{X_i}, \mathcal{D}_{X_i})$ for every $i \in I$ and let $X=\prod_{i \in I} X_i$. Is it true that player two has a winning strategy in $G^\omega_1(\mathcal{C}_X, \mathcal{D}_X)$?
\end{question}

Question $\ref{questprod}$ was addressed by Aurichi in \cite{A}, who gave the following partial answer (note that a countable $\pi$-base suggests an obvious winning strategy for player two).

\begin{theorem} \label{aurthm}
(Aurichi, \cite{A}) Let $\{X_i: i \in I \}$ be a family of spaces such that $X_i$ a countable $\pi$-base for every $i \in I$ and $X=\prod_{i \in I} X_i$ . Then player two has a winning strategy in $G^\omega_1(\mathcal{C}_X, \mathcal{D}_X)$.
\end{theorem}

We will be able to replace countable $\pi$-weight with countable $\pi$-character in the above result, by proving that a winning strategy for player two is actually equivalent to countable $\pi$-weight in the realm of spaces with countable $\pi$-character. This will be the byproduct of a game-theoretic proof of Shapirovskii's bound on the regular open sets, which stands on the playful characterization of cellularity provided by Theorem $\ref{naturalprop}$

\begin{definition}
We define the \emph{cellular-open number} of $X$ ($con(X)$) to be the least cardinal $\kappa$ such that for every play of player one, player two is able to win in $\alpha$ moves with $|\alpha| \leq \kappa$.
\end{definition}

\begin{theorem} \label{naturalprop} 
For every space $X$, $c(X)=con(X)$.
\end{theorem}

\begin{proof}
To prove that $con(X) \leq c(X)$, suppose $c(X)=\kappa$. Let $\mathcal{M}_\alpha$ be the maximal cellular family played by player one at the $\alpha$th inning and suppose that in his turn player two picks $U_\alpha \in \mathcal{M}_\alpha$.

Moreover, suppose that we picked open sets $ \{V_\alpha: \alpha < \beta \}$ such that $\{U_\alpha \cap V_\alpha: \alpha < \beta \}$ is a pairwise disjoint family of non-empty open sets. If $\bigcup_{\alpha < \beta} U_\alpha$ is dense then player two won. Otherwise, we can find a non-empty open set $V_\beta$ such that $V_\beta \cap \bigcup_{\alpha < \beta } U_\alpha=\emptyset$. Now let $U_\beta \in \mathcal{M}_\beta$ be such that $V_\beta \cap U_\beta \neq \emptyset$. In the $\beta$-th inning, player two picks $U_\beta $. Note that the family $\{U_\alpha \cap V_\alpha: \alpha \leq \beta \}$ is pairwise disjoint. If the two players were able to carry this on for $\kappa^+$ many innings we would obtain a $\kappa^+$-sized pairwise disjoint family of open sets, which contradicts $c(X) =\kappa$. Thus $con(X) \leq \kappa$.

Viceversa, assume that $con(X)=\kappa$ and suppose by contradiction that $c(X) \geq \kappa^+$. Then we can fix a maximal cellular family $\mathcal{C}$ of size at least $\kappa^+$. Let player one play $\mathcal{C}$ at every winning. By $con(X) =\kappa$, player two is able to win the game in $\alpha$ many innings, with $|\alpha| \leq \kappa$, and thus the fact that $\mathcal{C}$ is a cellular family is contradicted.
\end{proof}

\begin{theorem} \label{gamebound}
Let $X$ be an almost regular space. 

\begin{enumerate}

\item \label{shap} $\pi w(X) \leq (\pi \chi(X))^{c(X)}$.
\item \label{aur} If player two has a winning strategy in $G^\omega_1(\mathcal{C}, \mathcal{D})$ and $\pi \chi(X) \leq \omega$ then $X$ has countable $\pi$-weight.
\end{enumerate}
\end{theorem}

\begin{proof}
Set $\pi \chi(X)=\kappa$ and $c(X)=\mu$. For every $x \in X$ let $\mathcal{P}(x)$ be a local $\pi$-base at $x$ having size $\kappa$. By Lemma $\ref{naturalprop}$ we can fix a winning strategy $F$ for player two in the game $G^{\mu^+}_1(\mathcal{C}_X, \mathcal{D}_X)$. 

\noindent {\bf Claim 1.}
Let $\alpha < \mu^+$ and $(\mathcal{U}_\beta: \beta < \alpha )$ be a sequence of maximal cellular families. For every neighbourhood $U$ of $x$ there is a non-empty open $V \subset U$ with $V \in \mathcal{P}(x)$ and a maximal cellular family $\mathcal{U}$ such that:

$$F((\mathcal{U}_\beta: \beta < \alpha )^\frown (\mathcal{U})) \subset V$$

\begin{proof}[Proof of Claim 1]

Suppose this is not true. Then for every $x \in X$ fix a neighbourhood $U_x$ of $x$, such that for every open $V \subset U_x$ with $V \in \mathcal{P}(x)$ and every maximal cellular family $\mathcal{U}$ we have:

$$F((\mathcal{U}_\beta: \beta < \alpha)^\frown (\mathcal{U})) \nsubseteq V$$

Let $\mathcal{P}=\{B : (\exists x)(B \in \mathcal{P}(x) \wedge B \subset U_x) \}$. The poset of all cellular families refining $\mathcal{P}$ satisfies the assumptions of Zorn's Lemma, and hence we can fix a maximal element $\mathcal{U}$ in it. Since $\mathcal{P}$ has dense union in $X$, $\mathcal{U}$ is actually a maximal cellular family. 

Then $F((\mathcal{U}_\beta: \beta < \alpha)^\frown (\mathcal{V})) \subset B$, for some $B \in \mathcal{P}(x)$ such that $B \subset U_x$. But this is a contradiction.

 \renewcommand{\qedsymbol}{$\triangle$}
\end{proof}

For every $x \in X$, let $\tau(x)=\{U \in \tau: x \in U \}$.

Choose a point $x_\emptyset \in X$ satisfying Claim 1 for the empty sequence. For every $U \in \tau(x_\emptyset)$ we can then choose an open set $V^0_U \in \mathcal{P}(x_\emptyset)$ such that $V^0_U \subset U$ and a maximal cellular family $\mathcal{U}^0_U$ such that $F((\mathcal{U}^0_U)) \subset V_U$. Let $\{U_\gamma(x_\emptyset): \gamma < \kappa \}$ enumerate the set $\{V^0_U: U \in \tau(x_\emptyset) \}$. For every $\alpha < \kappa$ choose $U \in \tau(x_\emptyset)$ such that $F((\mathcal{U}^0_U)) \subset U_\alpha(x_\emptyset)$ and define $\mathcal{U}_{\{(0, \alpha)\}}=\mathcal{U}^0_U$.

For each $\alpha_0 < \kappa$ choose $x_{\{(0, \alpha_0)\}} \in X$ satisfying Claim 1 for the sequence $(\mathcal{U}_{\{(0, \alpha_0)\}})$. Then for each $U \in \tau(x_{\{(0, \alpha_0)\}})$ we can choose an open set $V^1_U \in \mathcal{P}(x_{\{(0, \alpha_0)\}})$ and a maximal cellular family $\mathcal{U}^1_U$ such that $F((\mathcal{U}_{\{(0, \alpha_0)\}}, \mathcal{U}^1_U)) \subset V^1_U$. Let $\{U_\gamma(x_{\{(0, \alpha_0)\}}): \gamma < \kappa \}$ enumerate the set $\{V^1_U: U \in \tau (x_{\{(0, \alpha_0)\}}) \}$. For every $\alpha < \kappa$ choose $U \in \tau(x_{\{(0, \alpha_0)\}})$ such that $F((\mathcal{U}_{\{(0, \alpha_0)\}}, \mathcal{U}^1_U)) \subset U_\alpha(x_{\{(0, \alpha_0)\}})$ and define $\mathcal{U}_{\{(0, \alpha_0), (1, \alpha)\}}=\mathcal{U}^1_U$.

Now let $\theta < \mu^+$ and suppose that for every $f \in \bigcup_{\beta< \theta} \kappa^\beta$ we have chosen a point $x_f \in X$, a maximal cellular family $\mathcal{U}_f$ and a local $\pi$-base $\{U_\gamma(x_f): \gamma < \kappa \}$ at $x_f$ such that for all $\beta < \theta$ and for every $f \in \kappa^\beta$ and $\rho< \kappa$ we have:

$$F((\mathcal{U}_{f \upharpoonright \gamma}: \gamma < \beta )^\frown (\mathcal{U}_{f \cup \{(\beta, \rho)\}})) \subset U_\rho(x_f).$$

Fix $f \in \kappa^\theta$. Let $x_f \in X$ be the point guaranteed by applying Claim 1 to the sequence $(\mathcal{U}_{f \upharpoonright \beta}: \beta < \theta)$ of maximal cellular families. Then, for each $U \in \tau(x_f)$ we can choose an open set $V^\theta_U \in \mathcal{P}(x_f)$ and a maximal cellular family $\mathcal{U}^\theta_U$ such that $F((\mathcal{U}_{f \upharpoonright \beta}: \beta < \theta)^\frown (\mathcal{U}^\theta_U)) \subset V^\theta_U$. Let $\{U_\gamma(x_f): \gamma < \kappa \}$ enumerate $\{V^\theta_U: U \in \tau(x_f)\}$. For every $\alpha < \kappa$, choose $U \in \tau(x_f)$ such that $F((\mathcal{U}_{f \upharpoonright \beta}: \beta <\theta)^\frown (\mathcal{U}^\theta_U)) \subset U_\alpha(x_f)$ and define $\mathcal{U}_{f^\frown ((\theta, \alpha))}=\mathcal{U}^\theta_U$.

At the end of the induction we will have chosen, for each $f \in \bigcup_{\theta < \mu^+} \kappa^\theta$ a point $x_f \in X$, a maximal cellular family $\mathcal{U}_f$ and a local $\pi$-base $\{U_\gamma(x_f): \gamma < \kappa \}$ at $x_f$ such that for each $\alpha< \kappa$, $\chi < \mu^+$ and $f \in \kappa^\chi$ we have:

$$F((\mathcal{U}_{f \upharpoonright \beta}: \beta < \chi)^\frown (\mathcal{U}_{f \cup \{(\chi, \alpha)\}})) \subset U_\alpha(x_f).$$

Let $D=\{x_f: f \in \bigcup_{\alpha < \mu^+} \kappa^\alpha\}$.

\vspace{.1in}

\noindent {\bf Claim 2.} The set $D$ is dense in $X$.

\begin{proof}[Proof of Claim 2]
Suppose not. Then we can find an open set $V$ such that $\overline{V} \cap D=\emptyset$. Since $\{U_\gamma(x_\emptyset): \gamma < \kappa \}$ forms a local $\pi$-base at $x_\emptyset$ we can choose an ordinal $\gamma_0<\kappa$ such that $U_{\gamma_0}(x_\emptyset) \cap \overline{V}=\emptyset$. Player one then plays $\mathcal{U}_{\{(0, \gamma_0)\}}$ in his first move. Suppose player I has played maximal cellular families $( \mathcal{U}_{f \upharpoonright \beta}: \beta < \alpha)$ for some $f \in \kappa^\alpha$ such that $f(0)=\gamma_0$. Choose $\gamma_\alpha$ such that $U_{\gamma_\alpha}(x_f) \cap \overline{V}=\emptyset$ and let player one play $\mathcal{U}_{f \cup \{(\alpha, \gamma_\alpha)\}}$. In this way we build a play which is lost by player two using the strategy $F$. But this contradicts the fact that $F$ is a winning strategy for player two.
 \renewcommand{\qedsymbol}{$\triangle$}
\end{proof}

Since $|D| \leq \kappa^\mu$, we then have that $d(X) \leq \kappa^\mu$. Now we have $\pi w(X) \leq \pi \chi(X) \cdot d(X)  \leq (\pi \chi(X))^{c(X)}$.

To prove the second statement, note that if player two has a winning strategy in $G^\omega_1(\mathcal{C}, \mathcal{D})$ then the set $D$ defined before Claim 2 is actually countable. Now, every separable space of countable $\pi$-character has countable $\pi$-weight.
\end{proof}

\begin{corollary}
(Shapirovskii) Let $X$ be an almost regular space. Then $\rho(X) \leq \pi \chi(X)^{c(X)}$.
\end{corollary}

\begin{proof}
This follows from Theorem $\ref{gamebound}$, ($\ref{shap})$ and the inequality $\rho(X) \leq \pi w(X)^{c(X)}$, a proof of which can be found in \cite{J}.
\end{proof}

Theorem $\ref{gamebound}$, ($\ref{aur}$) yields the following characterization.

\begin{corollary} \label{equiv}
Let $X$ be a regular space of countable $\pi$-character. Then $X$ has countable $\pi$-weight if and only if player II has a winning strategy in $G^\omega_1(\mathcal{C}, \mathcal{D})$.
\end{corollary}

Putting together Corollary $\ref{equiv}$  and Theorem $\ref{aurthm}$ we obtain the following partial answer to Question $\ref{questprod}$.

\begin{theorem}
Having a winning strategy for player two in the game $G^\omega_1(\mathcal{C}, \mathcal{D})$ is productive for the class of spaces of countable $\pi$-character.
\end{theorem}

Hajnal and Juh\'asz proved in \cite{HJ} that the cardinality of first-countable ccc Hausdorff spaces is at the most the continuum. Combining Proposition $\ref{link}$, Theorem $\ref{weaklygamethm}$ and Theorem $\ref{naturalprop}$ we obtain that the Hajnal-Juh\'asz bound is true for the class of almost regular spaces. Note that there are almost regular non-Hausdorff spaces (Example $\ref{sakaiexample}$ provides an example of such a space).

\begin{theorem}
Let $X$ be an almost regular ccc first-countable space. Then $|X| \leq \mathfrak{c}$.
\end{theorem}

However, Hajnal and Juh\'asz's bound fails for first-countable $T_1$ spaces where player two has a winning strategy for the game $G^\omega_1(\mathcal{C}, \mathcal{D})$, as the following example shows.

\begin{example} \label{exampleT1}
Let $\kappa$ be any uncountable cardinal. There is a $T_1$ first-countable space such that $|X|=\kappa$ where Player II has a winning strategy for $G^\omega_1(\mathcal{C}_X, \mathcal{D}_X)$.
\end{example}

\begin{proof}
Let $A \subset \mathbb{R} \setminus \mathbb{Q}$ be a countable dense subset in the Euclidean topology. We define a topology on $X=(\mathbb{Q} \times \kappa) \cup A$, by declaring

$$\{(r, \beta): |r-q|<1/n , r \in \mathbb{Q}, \beta < \kappa \} \cup \{s \in A: |s-q|<1/n \} \setminus F$$

to be a basic open set, where $q \in \mathbb{Q}$ and $F \in [X]^{<\omega}$.

This defines a $T_1$ first countable topology on $X$. To see why Player two has a winning strategy for $G^\omega_1(\mathcal{C}_X, \mathcal{D}_X)$, let $\{U_n: n <\omega \}$ be a countable base for $A$ in the relative topology (which coincides with the Euclidean topology on $A$). Let $\mathcal{C}_n$ be the maximal cellular family player by player I at the $n$-th inning. Then $\{C \cap A: C \in \mathcal{C}_n \}$ is a maximal cellular family in $A$. So let $C_n \in \mathcal{C}_n$ be such that $U_n \cap C_n \neq \emptyset$. Player II plays $C_n$ at the $n$-th inning. It is clear from the definition of the topology on $X$ that $\bigcup \{C_n: n < \omega \}$ is dense in $X$.
\end{proof}

\end{document}